\documentclass[12pt,reqno]{amsart}
\usepackage{hyperref}
\usepackage[T1]{fontenc}
\usepackage[latin1]{inputenc}
\usepackage{latexsym}
\usepackage[arrow, matrix, curve]{xy}
\usepackage[dvips]{graphicx}
\usepackage{epsfig}
\usepackage{bm}
\usepackage{stmaryrd}
\usepackage{amssymb}
\usepackage{amsmath}
\usepackage{verbatim}
\usepackage{amsthm}
\usepackage{enumerate}
\usepackage{mathrsfs}
\usepackage[hmargin=3.5cm,vmargin=3.5cm]{geometry}
\numberwithin{equation}{section}

\renewcommand{\contentsname}{Contents\\{\footnotesize\normalfont(A table
of contents should normally not be included)}}

\newtheorem{theorem}{Theorem}[section]
\newtheorem{lemma}[theorem]{Lemma}
\newtheorem{corollary}[theorem]{Corollary}

\theoremstyle{definition}

\def\ba{{\mathbf a}}

\def\bk{{\mathbf k}}

\def\bp{{\mathbf p}}

\def\calA{{\mathcal A}}

\def\N{\mathbb N}

\def\Q{\mathbb Q}
\def\R{\mathbb R}

\def\Z{\mathbb Z}

\def\alp{{\alpha}}  
\def\bet{{\beta}}
\def\gam{{\gamma}} \def\Gam{{\Gamma}} 
\def\del{{\delta}} \def\Del{{\Delta}}
\def\eps{\varepsilon}

\def\tet{{\theta}}  \def\Tet{{\Theta}}

\def\balp{{\boldsymbol \alpha}}
\def\bbet{{\boldsymbol \beta}}

\def\d{{\partial}}

\def\le{\leqslant}
 \def\ge{\geqslant}

\def\d{{\,{\rm d}}}
\def\mmod#1{\;(\mathrm{mod}\;{#1})}

\DeclareMathOperator{\lcm}{lcm}

\DeclareMathOperator{\mes}{meas}

\def\llbracket{\lbrack\;\!\!\lbrack} \def\rrbracket{\rbrack\;\!\!\rbrack}

\begin{document}

\title[On generating functions in additive number theory, II]{On generating functions in additive number theory, II: Lower-order terms and applications to PDEs}
\author{J. Brandes}
\address{JB: Mathematical Sciences, University of Gothenburg and Chalmers Institute of Technology, 412 96 G\"oteborg, Sweden.}\email{brjulia@chalmers.se}

\author{S. T. Parsell}
\address{STP: Department of Mathematics, West Chester University, West Chester, PA 19383, USA.}\email{sparsell@wcupa.edu}

\author{C. Poulias}
\address{CP: School of Mathematics, University Walk, Clifton, Bristol BS8 1TW, UK.}\email{kp17312@bristol.ac.uk}

\author{G. Shakan}
\address{GS: Mathematical Institute, University of Oxford, Andrew Wiles Building, Radcliffe Observatory Quarter, Woodstock Road, Oxford, OX2 6GG, UK.}\email{george.shakan@gmail.com}

\author{R. C. Vaughan}
\address{RCV: Department of Mathematics, Pennsylvania State University, University Park, PA 16802, USA.}\email{rcv4@psu.edu}

%

\begin{abstract}
	We obtain asymptotics for sums of the form
	$$
		\sum_{n=1}^P e(\alpha_kn^k + \alpha_1n),
	$$
	involving lower order main terms. As an application, we show that for almost all $\alp_2 \in [0,1)$ one has
	$$
		\sup_{\alp_1 \in [0,1)} \Big| \sum_{1 \le n \le P} e(\alp_1(n^3+n) + \alp_2 n^3) \Big| \ll P^{3/4 + \eps},
	$$
	and that in a suitable sense this is best possible. This allows us to improve bounds for the fractal dimension of solutions to the Schr\"odinger and Airy equations.
\end{abstract}
\maketitle


\section{Introduction}

\noindent Exponential sums are a ubiquitous tool throughout analytic number theory, and have been studied in their own right at least since the 1920s. When $\bk = (k_1, \ldots, k_t)$ is a tuple of pairwise distinct natural numbers and $P$ is a large positive integer, the Weyl sum of multidegree $\bk$ is given by
\begin{align}\label{*1.1}
	f_{\bk}(\alp_{k_1}, \ldots, \alp_{k_t}) = \sum_{n=1}^P e(\alp_{k_1} n^{k_1} + \ldots + \alp_{k_t} n^{k_t}).
\end{align}
Such sums regularly feature in applications of the Hardy-Littlewood circle method in connection with diophantine systems of the shape
\begin{align}\label{*1.2}
	x_1^{k_j} + \ldots + x_s^{k_j} = y_1^{k_j}+ \ldots + y_s^{k_j} \qquad (1 \le j \le t).
\end{align}

Whilst the theory of systems of the kind \eqref{*1.2} has recently seen significant advances in the work of Wooley \cite{W:mvmt3, W:nec} and Bourgain, Demeter and Guth \cite{BDG} on Vinogradov's mean value theorem, our grasp of the cases involving lacunary degrees remains insufficient. One of the simplest such systems is that corresponding to $\bk = (1,3)$, a variant of which is given by
\begin{align}\label{*1.3}
	x_1^{3} + \ldots + x_s^{3} =x_1 + \ldots + x_s = 0.
\end{align}
Although recent progress on this system has been achieved by Br\"udern and Robert \cite{BR} and Wooley \cite{W:15}, a full understanding of the system \eqref{*1.3} remains tantalisingly out of reach. In both papers, the authors apply the circle method in order to derive asymptotic formul{\ae} for the number of solutions of such systems, and they succeed in doing so as soon as $s \ge 10$.  On the basis of standard heuristics, one would expect to be able to extend the range in which formul{\ae} of this kind are valid to at least $s \ge 9$, but unfortunately we lack a sufficiently precise understanding of the underlying Weyl sum $f_{1,3}(\balp)$ to achieve such a bound. A similar phenomenon occurs in forthcoming work of Hughes and Wooley \cite{HW}, which deals with moments of a weighted version of $f_{1,3}(\balp)$. Trying to make some headway towards a better understanding of these exponential sums is the main motivation behind the paper at hand.

The main motif underpinning the Hardy-Littlewood method is that sums of the shape \eqref{*1.1} should be small unless all components of the coefficient vector $\balp$ lie in the vicinity of fractions with a small denominator, in which case they can be well approximated by certain generating functions that are easier to handle and encode the adelic information inherent in the associated system \eqref{*1.2}.  To make this notion precise, we introduce some notation. Suppose that the entries of $\balp$ have a rational approximation of the shape
\begin{align}\label{1.4}
	\alp_{k_j} = a_{k_j}/q + \beta_{k_j} \qquad (1 \le j \le t)
\end{align}
with a common denominator $q$ satisfying $\gcd(q, a_{k_1}, \ldots, a_{k_t})=1$, and define
\begin{align*}
	S_{\bk}(q;\ba) = \sum_{x=1}^q e\left(q^{-1}\sum_{j=1}^t a_{k_j} x^{k_j}\right) \quad \text{and} \quad I_{\bk}(\bbet) = \int_0^P e \left(\sum_{j=1}^t \bet_{k_j} x^{k_j}\right) \d x.
\end{align*}
In this notation, we anticipate that $q^{-1}S_{\bk}(q; \ba)I_{\bk}(\bbet)$ should be a good approximation to $f_{\bk}(\balp)$, and we denote the difference by
\begin{align}\label{1.5}
	\Del_{\bk}(q, \ba; \bbet) = f_{\bk}(\ba/q + \bbet) - q^{-1} S_{\bk}(q; \ba) I_{\bk}(\bbet).
\end{align}

There is a considerable body of work related to Weyl sums of the type \eqref{*1.1} and their approximations \eqref{1.5}. When $t=1$ so that $\bk = k$, it is known from \cite[Theorem~4.1]{V:HL} that
\begin{align}\label{1.6}
\Del_k(q,a;\bet) \ll q^{1/2+\eps}(1+P^k |\bet|)^{1/2},
\end{align}
and Daemen \cite[Theorem~2]{Dae} and Br\"udern and Daemen \cite[Theorem~1]{BD} showed that this bound is sharp up to at most a factor of $q^\eps$. For general multidegrees $\bk$ we have the weaker bound
\begin{align*}
	\Del_{\bk}(q, \ba; \bbet) \ll q \left(1+ \sum_{j=1}^t P^{k_j} |\bet_{k_j}| \right)
\end{align*}
from \cite[Theorem~7.2]{V:HL}. This bound has been improved for $\bk = (1,k)$ by Br\"udern and Robert \cite[Theorem~3]{BR}, who obtained the estimate
\begin{align}\label{1.7}
	\Del_{1,k}(q, \ba; \bbet) \ll q^{1-1/k+\eps}(1+P^k |\bet_k|)^{1/2}.
\end{align}
Whilst their result holds for all $k \ge 2$, an epsilon-free version is available in the quadratic case due to the last author \cite[Theorem~8]{V:09}.
This invites the question of how optimal is the bound in \eqref{1.7}.

The primary objective of this memoir is to examine the exponential sum $f_{1,k}(\alp_1, \alp_k)$ and its associated error term $\Del_{1,k}(q,\ba;\bbet)$ more closely. Our main result is the following.
\begin{theorem}\label{T1.1}
	Let $k \ge 2$. Assume \eqref{1.4} with $(q, a_k)=1$ and $|\bet_1| \le (2q)^{-1}$.  Then
	\begin{align*}
		f_{1,k}(\alp_1, \alp_k) &= q^{-1} \sideset{}{^\dag}\sum_{d|q} S_{1,k}(q; d \llbracket a_1/d \rrbracket, a_k) I_{1,k}\left(\alp_1-\frac{\llbracket a_1/d \rrbracket}{q/d}, \bet_k\right)\\
		&\qquad + O(q^{1/2+\eps}(1+|\bet_k|P^k)^{1/2}\log P),
	\end{align*}
	where $\llbracket x \rrbracket$ denotes a closest integer to $x$, and the notation $\sum^\dagger$ indicates that the sum runs over all distinct values of $d\llbracket a_1/d \rrbracket$ satisfying $(\llbracket a_1/d \rrbracket, q/d)=1$.
	
	If, moreover, we have
	\begin{align}\label{1.8}
		|\bet_k| \le (4kqP^{k-1})^{-1},
	\end{align}
	then the error term in the above asymptotic may be replaced by $O(q^{1/2 +\eps})$.
\end{theorem}
Thus, by extracting additional main terms, we are able to obtain an error term of the same quality as in \eqref{1.6}, which is essentially optimal. We note that the factor $\log P$ in the error term can be eliminated by means of a more careful analysis. Observe also that the coprimality condition $(\llbracket a_1/d \rrbracket, q/d)=1$ implies that the fractions $\llbracket a_1/d\rrbracket / (q/d)$ are reduced and therefore pairwise distinct.

In the case when $k=3$, it follows from Dirichlet's approximation theorem that every $\alp \in [0,1]$ has an approximation $\alp=a/q + \bet$ with $q(1+P^3 |\bet|) \le 2 P^{3/2}$. Thus, in the cubic case we obtain the following.
\begin{corollary}\label{C1.2}
	Assume \eqref{1.4} with $q(1+P^3 |\bet_3|) \le 2 P^{3/2}$, where $(q, a_3)=1$ and $|\bet_1| \le (2q)^{-1}$.
	\begin{enumerate}[(a)]
		\item \label{C1.2a}
		We have
		\begin{align*}
			f_{1,3}(\alp_1, \alp_3) &= q^{-1} \sideset{}{^\dag}\sum_{d|q} S_{1,3}(q; d \llbracket a_1/d \rrbracket, a_3) I_{1,3}\left(\alp_1-\frac{\llbracket a_1/d \rrbracket}{q/d}, \bet_3\right) + O(P^{3/4+\eps}).			
		\end{align*}
		\item \label{C1.2b}
		Moreover, we have the bound
		\begin{align*}
			f_{1,3}(\alp_1, \alp_3) \ll \frac{P^{1+\eps}}{(q+q|\bet_3|P^3)^{1/3}} + P^{3/4+\eps}
		\end{align*}
		uniformly in $\alp_1$.
	\end{enumerate}
\end{corollary}
For general degree $k$, an analogous chain of reasoning would replace the error terms $O(P^{3/4+\eps})$ by $O(P^{k/4+\eps})$, which is trivial when $k \ge 4$. Thus, our result in Theorem~\ref{T1.1} is strongest in the cubic case, and for higher degrees should be viewed as a bound for the major arcs only.

When $d=(a_1,q)$ we have $d \llbracket a_1/d \rrbracket = a_1$; this is the leading term in Theorem~\ref{T1.1} and corresponds to the approximation in \eqref{1.5}. We can thus rephrase the conclusion of Theorem~\ref{T1.1} in the form
\begin{align}\label{1.9}
	\Del_{1,k}(q,\ba;\bbet) &= q^{-1}\sideset{}{^\dag}\sum_{\substack{d|q \\ d \neq (q, a_1)}}  S_{1,k} (q; d \llbracket a_1/d\rrbracket, a_k)I_{1,k}\left(\alp_1-\frac{\llbracket a_1/d \rrbracket}{q/d}, \bet_k\right) \nonumber\\
	& \qquad + O(q^{1/2+\eps}(1+|\bet_k|P^k)^{1/2}\log P).
\end{align}
When $d \neq (a_1,q)$ the fractions $a_1/d$ are all non-integral. In particular, when $a_1/d$ is half an odd integer, there are two choices for $\llbracket a_1/d \rrbracket$ and both may occur in the sum. Note further that the sum in \eqref{1.9} is empty if and only if $a_1$ is a multiple of $q$. When $q \nmid a_1$ we have $(a_1,q) \neq q$, and thus it will always contain at least the term $S_{1,k}(q;0,a_k) I_{1,k}(\alp_1, \bet_k)$ corresponding to $d=q$, and in many cases this is the only one. For instance, it is not hard to see that when $a_1= 1$ and $q > 1$ is odd, the sum in Theorem~\ref{T1.1} contains precisely the terms corresponding to $d=1$ and $d=q$, and so in this case the asymptotic formula reads
\begin{align*}
	f_{1,k}(\alp_1, \alp_k) &= q^{-1} \big(S_{1,k}(q;1, a_k) I_{1,k}(\bet_1, \bet_k) + S_{1,k}(q;0, a_k) I_{1,k}(\alp_1, \bet_k)\big)\\
	& \qquad + O(q^{1/2+\eps}(1 + |\bet_k|P^k)^{1/2}\log P).
\end{align*}
This behaviour, which occurs in many generic situations, indicates that we cannot expect the secondary terms to be subject to any significant cancellation. In fact, we have the following result on the size of the error term.
\begin{theorem}\label{T1.3}
	Suppose that $a_k$ and $q$ are coprime integers satisfying $|\alp_k - a_k/q| < q^{-2}$, and assume that \eqref{1.4} holds with $|\bet_1| < (2q)^{-1}$.
	\begin{enumerate}[(a)]
		\item \label{T1.3a}
We have the upper bound			
\begin{align*}
				\Del_{1,k}(q,\ba;\bbet) \ll q^{1/2+\eps}\left(1+\sideset{}{^\dag}\sum_{\substack{d|q\\d \neq (q,a_1)}}  \frac{|S_{1,k}(q; d\llbracket a_1/d \rrbracket, a_k)|}{q^{1/2}} \right) (1+|\bet_k|P^k)^{1/2} \log P.
			\end{align*}
		\item \label{T1.3b}
			Suppose now that $\bet_k = 0$ and $a_1=1$ with $q=p^k$, where $p$ is an odd prime. Then, whenever $\| \alp_1 P \| \ge \delta$ for some suitable real number $\del>0$, we have the lower bound
			\begin{align*}
				|\Del_{1,k}(q, 1, a_k; \bet_1,0)| \ge \frac{4\del}{3\pi} q^{1-1/k}+ O(q^{1/2+\eps} \log P).
			\end{align*}
	\end{enumerate}
\end{theorem}
Thus Theorem~\ref{T1.3} shows that the bound \eqref{1.7} of Br\"udern and Robert is sharp at least when $\bet_k$ is small and $q$ is a perfect $k$-th power of a square-free number. Here, the term of size $q^{1-1/k} = p^{k-1}$ arises from the exponential sum $S_{1,k}(q;0,a_k)$ via Lemma~4.4 in \cite{V:HL}, and as noted above, unless $q|a_1$ this term will always appear in the sum in \eqref{1.9}. Thus, large values of $\Del_{1,k}(q,\ba;\bbet)$ cannot be considered an exceptional occurrence when $\bet_k$ is small.

%

As a consequence of Theorem~\ref{T1.1}, we are able to make progress on a problem concerning the fractal dimension of solutions of certain partial differential equations. The motivation for this problem goes back to optical experiments by Talbot \cite{Tal} in the 1830s concerning the diffraction of light passing through a grating. Berry \cite{Ber} later initiated the theoretical investigation of the problem and has in particular made predictions regarding the fractal dimension of the diffraction pattern along certain slices in space. The reader is referred to Chapter 2 of \cite{ET} and the introduction of \cite{ES} for an introduction to the general topic as well as a more thorough history of this particular problem.

In this paper, we shall focus in particular on the family of partial differential equations given by
\begin{align}\label{1.10}
	i \partial_t  q(t,x) - i^k \partial^{(k)}_x q(t,x) = 0 \qquad (k \in \N),
\end{align}
where $t \in \R$ and $x \in \R / 2 \pi \Z$. When $k=2$, the reader will recognize \eqref{1.10} as the linear Schr\"odinger equation, while the case $k=3$ corresponds to the linear part of the Korteweg-de~Vries (KdV) equation, also known as Airy's equation. For any $k$, given initial data $g_k(n) \in L^2(\R/2 \pi \Z)$, the evolution of $g_k$ under \eqref{1.10} is given by
\begin{align*}
	q_k(t,x) = \sum_{n \in \Z} \hat g_k(n) e^{i t n^k + i x n}.
\end{align*}
Clearly, $q_k$ is periodic in both $t$ and $x$ with period $2 \pi$.

We are interested in the restriction of $q_k$ to linear subsets of $(\R / 2 \pi \Z) \times (\R / 2 \pi \Z)$.
Given $c \in \R$ and $r \in \Q \setminus \{0\}$, as well as initial data $g_k$, let
\begin{align*}
	q_{k;r,c}(x) = \sum_{n \in \Z} \hat g_k(n) e^{i (c-rx)n^k + i x n}
\end{align*}
denote the restriction of $q_k$ to the oblique line $t+ rx=c$. 
Recall that the fractal  (also known as upper Minkowski or upper box-counting) dimension of a bounded set $E$ is given by
$$
	\overline{\text{dim}}(E)=\limsup_{\eps\to 0}\frac{\log({\mathcal N}(E,\eps))}{\log(1/\eps)},
$$
where ${\mathcal N}(E,\eps)$ is the minimum number of $\eps$-balls required to cover $E$.
Assuming that $g_k$ is a suitably well-behaved function, we would like to know the the fractal dimension of the real and imaginary parts of the graph of $q_{k;r,c}$ for a typical $c$. Note here that it is possible for either the real or the imaginary part of the graph to vanish, so we are really interested in the size of the larger of the two.
The simplest non-constant choices for $g_k$ are step functions, and in such situations, we see that in order to make progress, it is imperative to understand the distribution of large values of exponential sums. As it is convenient to work with dyadic sums in this context, we modify our definition \eqref{*1.1} by writing
\begin{align}\label{1.11}
	f_{\bk}(\alp_{k_1}, \ldots, \alp_{k_t}; Q) = \sum_{Q < n \le 2 Q} e(\alp_{k_1} n^{k_1} + \ldots + \alp_{k_t} n^{k_t}),
\end{align}
where $Q$ is a positive number.
Let $\Tet_k$ denote the set of all $\tet \in \R$ such that for almost all $\gam \in [0,1)$ one has
\begin{align}\label{1.12}
	\sup_{Q \ge 1} Q^{-\tet} \sup_{\alp \in [0,1)} |f_{1,k}(\alp, \alp+\gam; Q)| \ll_{\tet,\gam} 1,
\end{align}
and set $\tet_k = \inf \Tet_k$. The size of $\theta_k$ and related quantities has recently been studied by Chen and Shparlinski \cite{ChS}, building on work by Wooley \cite{W:16}. Clearly, one sees that
\begin{align*}
	\lfloor 2 Q \rfloor - \lfloor Q \rfloor = \int_0^1 |f_{1,k}(\alp, \alp+\gam; Q)|^2 \d \alp \le \left(\sup_{\alp \in [0,1)}|f_{1,k}(\alp, \alp+\gam; Q)| \right)^2 \le Q^2,
\end{align*}
whence we have the trivial bounds
\begin{align}\label{1.13}
	1/2 \le \tet_k \le 1
\end{align}
valid for all $k \ge 2$ and for the entire range $\gam \in [0,1)$. Moreover, we have the trivial bound $\sup_{\alp, \gam}|f_{1,k}(\alp, \alp+\gam; Q)|=Q$, and it is known (see e.g. \cite[Corollary~2.2]{ChS}) that for independent variables $\gam$, $\alp$ we have $|f_{1,k}(\alp, \alp+\gam; Q)| \ll Q^{1/2 + \eps}$ almost everywhere. It turns out that in our case where only one of the variables is restricted to lie in the complement of a thin set while the other one ranges freely, the bound is appreciably larger.
\begin{theorem}\label{T1.4}
	We have $\tet_2 = \tet_3 = 3/4$.
\end{theorem}
Theorem~\ref{T1.4} may be a bit surprising as one naively expects square root cancellation in exponential sums. As will transpire from the proof, it turns out that for almost every $\gamma$ in \eqref{1.12} the supremum is obtained for a special choice of $\alpha$ on what can be considered a major arc.  One might speculate that $\tet_k = 3/4$ for all $k\ge 2$. Indeed, one might hope to adapt the proof of Theorem~\ref{T1.4} above to show that for almost all $\gam$ this gives the correct extremal value on a suitable set of major arcs and that for almost all $\gam$ the sum is smaller on the corresponding minor arcs. This latter speculation would be consistent with the main result of the last author and Wooley \cite{VW}.

With the help of Theorem~\ref{T1.4}, we can address our motivating problem.
\begin{theorem}\label{T1.5}
	For $k=2,3$ let $g_k$ be a step function, and fix $r \in \Q \setminus \{0\}$. Set $\alp_2 = 1/8$ and $\alp_3 = 1/12$. Then, for almost every $c \in \R$, the function $q_{k;r,c}$ satisfies the H\"older condition $C^{\alp}$ for every $\alp<\alp_k$. In particular, the fractal dimension of the graph of the real and imaginary parts of $q_{k;r,c}$ is at most $2-\alp_k$.
\end{theorem}

This improves the values of $\alp_2 = 1/10$ and $\alp_3 = 1/27$ of the fourth author with Erdo\u{g}an \cite[Theorem~1.1]{ES}. The proof is identical to that of \cite[Corollary 3.5]{ES}, but inputs our Theorem~\ref{T1.4} instead of \cite[Proposition 3.3]{ES}. Moreover, the results of Theorem~\ref{T1.5} can be transferred to non-linear partial differential equations, in particular the non-linear Schr\"odinger and KdV equations, by the same methods as Theorems 1.2 and 1.3 are derived from Theorem~1.1 in \cite{ES}. 

Note that Theorem~1.1 in \cite{ES} also gives a lower bound for the fractal dimensions in question. Specifically, the authors show that the graph of at least one of the real and imaginary parts of $q_{k;r,c}$ has fractal dimension of at least $2-1/(2k)$, and this is sharp at least in the Schr\"odinger case $k=2$. Moreover, they remark that, if it were true that $\tet_k = 1/2$, their argument could be adapted to show that this lower bound reflects the actual value. Our Theorem~\ref{T1.4} rules out this approach at least for the cases $k=2$ and $k=3$. Meanwhile, if our speculation that $\tet_k = 3/4$ could be substantiated for all $k$, it would imply that the maximum dimension of the respective graphs of the real and imaginary parts would lie in the range $[2-\frac{1}{2k}, 2-\frac{1}{4k}]$. It is worth noting that Lemma 2 of \cite{Os} along with \cite{ES} imply that for special combinations of initial data and oblique lines the fractal dimension in Theorem~\ref{T1.5} for $k=2$ is precisely 7/4  (see \cite[Footnote 3]{ES} for more details).

\textbf{Notation.} Throughout the paper, we make use of the following conventions. All statements involving the letter $\eps$ are claimed to be true for all (sufficiently small) $\eps>0$. Thus, the precise `value' of $\eps$ is allowed to change from one line to the next. Moreover, $P$ always denotes a large positive number. We use Vinogradov's and Landau's notation liberally, and here the implied constants are allowed to depend on $k$ and $\eps$, but never on $P$, $Q$ or $\balp$.

\section{Preliminary Lemmata}

\noindent In this section, we briefly collect some technical lemmata that will be of use in our arguments later. All of these results pertain to the case $\bk = (1,k)$, and in order to avoid clutter, we will in our arguments below drop the multidegree $(1,k)$ in our notation. Throughout, $Q$ denotes a positive number. For easier reference, we begin by stating a few results from the literature.
\begin{lemma}\label{L2.1}
	Let $a_1, a_k \in \Z$ and $q \in \N$, and suppose that $(a_k, q)=1$.
	\begin{enumerate}[(a)]
		\item \label{L2.1a}
			Uniformly in $a_1$, we have $S(q; a_1, a_k) \ll q^{1-1/k+\eps}$.
		\item \label{L2.1b}
			Moreover, $S(q; a_1, a_k) \ll q^{1/2+\eps}(q, a_1)$.	
	\end{enumerate}
\end{lemma}
\begin{proof}
	These are Theorem~7.1 and Lemma~4.1 in \cite{V:HL}, respectively.
\end{proof}
We also record an elementary average bound for exponential sums.
\begin{lemma}\label{L2.2}
	For any positive integer $q$ and any integer $a_k$ we have
	\begin{align*}
		\sum_{b=1}^q |S(q; b, a_k)| \le q^{3/2}.
	\end{align*}
\end{lemma}
\begin{proof}
	By the Cauchy-Schwarz inequality we see that
	\begin{align*}
		\sum_{b=1}^q |S(q; b, a_k)|  \le q^{1/2} \left(\sum_{b=1}^q |S(q; b, a_k)|^2\right)^{1/2},
	\end{align*}
	and expanding the square yields
	\begin{align*}
		\sum_{b=1}^q |S(q; b, a_k)|^2 = \sum_{b=1}^q  \sum_{x,y=1}^q e\biggl(\frac{b (x-y)+ a_k (x^k-y^k)}{q}\biggr)= q \sum_{x=1}^q 1 = q^2.
	\end{align*}
	This completes the proof.
\end{proof}

The next result is a direct consequence of \cite[Lemma~4.2]{V:HL}.
\begin{lemma}\label{L2.3}
	Suppose that $\phi$ is a twice continuously differentiable function on an interval $I$ and let $H\ge 2$ be a number such that $|\phi'(x)| \le H$ for all $x \in I$. Suppose further that $\phi''$ has at most finitely many zeros in the interval $I$. Then
	\begin{align*}
	\sum_{x \in I \cap \Z} e(\phi(n)) = \sum_{|h|\le H} \int_I e(\phi(x) - hx) \d x + O(\log H).
	\end{align*}
\end{lemma}
\begin{proof}
	This is immediate upon partitioning $I$ into subintervals on which $\phi'$ is monotonic, and then applying Lemma~4.2 of \cite{V:HL} on each of these finitely many intervals.
\end{proof}

We continue with bounds on oscillating integrals. For a measurable subset $\calA$  we denote
\begin{align*}
	I(\bet_1, \bet_k; \calA) = \int_{\calA} e(\bet_1 x + \bet_k x^k) \d x.
\end{align*}
We then have the following bounds for $I(\bet_1, \bet_k; \calA)$.
\begin{lemma}\label{L2.4}
	Let $k \ge 2$ and suppose that $\calA$ is a finite union of pairwise disjoint intervals.
	\begin{enumerate}[(a)]
		\item \label{L2.4a}
			Let $\tau>0$ be a parameter satisfying $|\bet_1 + k\bet_k x^{k-1}| \ge \tau$ for all $x \in \calA$. Then
			\begin{align*}
				I(\bet_1, \bet_k; \calA) \ll \tau^{-1}.
			\end{align*}
		\item \label{L2.4b}
			Assume that $\calA \subseteq [Q,2 Q]$ for some $Q>0$. Then, whenever $|\bet_k| \neq 0$ we have
			$$
				I(\beta_1, \beta_k; \calA) \ll (|\beta_k| Q^{k-2})^{-1/2}.
			$$
	\end{enumerate}
\end{lemma}
\begin{proof}
	These are Lemmata 4.2 and 4.4 in \cite{Titch}, respectively, applied to the function $F(x) = \bet_1 x + \bet_k x^k$.
\end{proof}

\begin{lemma}\label{L2.5}
	Assume that $k \ge 2$ and $\bet_1 \neq 0$. Suppose further that $\calA$ is a union of finitely many pairwise disjoint intervals contained inside $[Q, 2 Q]$ for some $Q>0$. Then we have the bound
	\begin{align*}
		I(\bet_1, \bet_k; \calA) \ll |\bet_1|^{-1} (1 + Q^k |\bet_k|)^{1/2}.
	\end{align*}
\end{lemma}
\begin{proof}
	Suppose first that the relation
	\begin{align}\label{2.1}
		|\bet_1 - k \bet_k x^{k-1}| \ge \textstyle{\frac12} |\bet_1|
	\end{align}
	holds for all $x \in \calA$. Then we see from Lemma~\ref{L2.4}\eqref{L2.4a} that
	\begin{align}\label{2.2}
		I(\bet_1, \bet_k; \calA) \ll |\bet_1|^{-1},
	\end{align}	
	which is sufficient to prove the lemma in this case. We may thus concentrate on the opposite case where the inequality \eqref{2.1} is violated for some $x \in \calA$. It follows from the triangle inequality that any such $x$ must satisfy the inequalities  $\frac12|\bet_1| \le k |\bet_k| x^{k-1} \le \frac32|\bet_1|$. Since $Q \le x \le 2 Q$, this can happen only if
	\begin{align}\label{2.3}
	|\bet_1| \asymp Q^{k-1}|\bet_k|
	\end{align}
	and in particular only when $\bet_k \neq 0$. We can thus deploy Lemma~\ref{L2.4}\eqref{L2.4b} and obtain
	\begin{align}\label{2.4}
		I(\bet_1, \bet_k; \calA) &\ll (Q^{k-2}|\bet_k|)^{-1/2} \notag \\ &\ll (|\bet_k| Q^{k-1})^{-1}(|\bet_k|Q^k)^{1/2}
		\ll |\bet_1|^{-1} (|\bet_k|Q^k)^{1/2},
	\end{align}	
	where in the last step we used \eqref{2.3} again. The full statement now follows upon combining \eqref{2.2} and \eqref{2.4}.
\end{proof}

\section{Proof of Theorem~\ref{T1.1}}\label{S3}

\noindent For the proof of our first main result it is convenient to work over dyadic ranges. Recalling our notation \eqref{1.11}, we make the analogous definition
\begin{align*}
	I(\bet_1, \bet_k; Q) = I_{1,k}(\bet_1, \bet_k;Q)= \int_Q^{2 Q} e(\bet_1 x + \bet_k x^k) \d x.
\end{align*}
Thus, if we can show that
\begin{align}\label{3.1}
	f(\alp_1, \alp_k;Q) &= q^{-1} \sideset{}{^\dag}\sum_{d|q} S(q; d \llbracket a_1/d \rrbracket, a_k) I\left(\alp_1-\frac{\llbracket a_1/d \rrbracket}{q/d}, \bet_k; Q\right) \nonumber\\
	&\qquad+ O(q^{1/2+\eps} (1+|\bet_k|Q^k)^{1/2}),
\end{align}
for any $Q \ge 1/2$, the conclusion of Theorem~\ref{T1.1} will follow upon dyadic summation, as
\begin{align*}
	f(\alp_1, \alp_k) = \sum_{i=1}^{\lceil \log P / \log 2 \rceil} f (\alp_1, \alp_k; 2^{-i}P).
\end{align*}

The initial stages of our argument follow along the lines of the proof of \cite[Theorem~3]{BR}, which in turn is an adaptation of the argument found in \cite[pp.~43--44]{V:HL}.
\begin{lemma}\label{L3.1}
	Assume \eqref{1.4} with $(a_1,q)=1$, and set
	\begin{align*}
		H=2^{k-1}q(1+kQ^{k-1}|\bet_k|).
	\end{align*}
	Then
	\begin{align*}
		f(\alp_1, \alp_k; Q)= q^{-1} \sum_{|h| \le H} S(q; a_1+h, a_k) I(\bet_1-h/q, \bet_k;Q) + O(q^{1/2}\log H).
	\end{align*}
\end{lemma}

\begin{proof}
	By sorting the terms of $f(\alp_1, \alp_k;Q)$ into congruence classes modulo $q$ and encoding the congruence condition in an exponential sum, we find that
	\begin{align}\label{3.2}
		f(\alp_1, \alp_k;Q) &= \sum_{r=1}^q e \left( \frac{a_1 r + a_k r^k}{q}\right) \sum_{\substack{Q < n \le 2 Q \\ n \equiv r \mmod q}} e(\bet_1 n + \bet_k n^k) \nonumber\\
		& =  \frac1q \sum_{-q/2 < b \le q/2} S(q; a_1+b, a_k)f(\bet_1 - b/q,\bet_k;Q).
	\end{align}
	We treat the sum $f(\bet_1 - b/q, \bet_k;Q)$ by Lemma~\ref{L2.3}, where we take $$
		\phi(x) = (\bet_1 - b/q)x + \bet_k x^k
	$$
and set
	\begin{align*}
		H_1 = 2^{k-1}(1+kQ^{k-1}|\bet_k|) - 1/2.
	\end{align*}
	Then we have $|\phi'(x)| \le H_1$ for all $x \le 2 Q$ and thus Lemma~\ref{L2.3} yields
	\begin{align*}
		f(\bet_1 - b/q, \bet_k;Q) = \sum_{|j| \le H_1} I(\bet_1-b/q-j, \bet_k;Q) + O(\log H_1).
	\end{align*}
	Using this within \eqref{3.2} and applying Lemma~\ref{L2.2} in the error term yields
	\begin{align*}
		f(\alp_1, \alp_k;Q) &= \frac{1}{q} \sum_{-q/2 < b \le q/2} \sum_{|j| \le H_1} S(q; a_1+ b+jq, a_k) I(\bet_1 - (b+jq)/q, \bet_k;Q) \\&\qquad+ O(q^{1/2}\log H_1).
	\end{align*}
	The proof is complete upon making the change of variables $b+qj=h$, noting that under the summation conditions this is in fact a bijection into the set of integers $h$ satisfying $-H < h  \le H$ where $H= q(H_1+1/2)$.
\end{proof}

We now distinguish two cases according to which term in $H$ is larger. Suppose first that $k|\bet_k|Q^{k-1}>1$, so that
\begin{align}\label{3.3}
	1 \ll H/q \ll |\bet_k|Q^{k-1}.
\end{align}
In such a situation, we discern from Lemma~\ref{L2.4}\eqref{L2.4b} and Lemma~\ref{L2.2} that
\begin{align*}
	\sum_{|h| \le H} S(q; a_1+h, a_k) I(\bet_1 -h/q, \bet_k;Q) &\ll (Q^{k-2}|\bet_k|)^{-1/2}  \left(\frac{H}{q}+1\right)  \sum_{a=1}^q |S(q; a, a_k)| \\
	&\ll q^{3/2}Q^{k/2}|\bet_k|^{1/2},
\end{align*}
where in the last step we used \eqref{3.3}.
Thus, in this situation, we find that
\begin{align}\label{3.4}
	f(\alp_1, \alp_k;Q) \ll q^{1/2}(Q^{k/2}|\bet_k|^{1/2} + \log H) \ll q^{1/2+\eps}Q^{k/2}|\bet_k|^{1/2},
\end{align}
which is satisfactory for the purposes of Theorem~\ref{T1.1}. 

It remains to study the behaviour of $f(\alp_1, \alp_k;Q)$ when $k |\bet_k| Q^{k-1} \le 1$, or in other words,
\begin{align}\label{3.5}
	H \ll q.
\end{align}
Set $d = (a_1+h, q)$ and $e =(a_1+h)/d$. In this notation, we have $h=de-a_1$ and $(e, q/d)=1$, and the conclusion of Lemma~\ref{L3.1} reads
\begin{align}\label{3.6}
	& f(\alp_1, \alp_k;Q)  \notag \\ &= q^{-1} \sum_{d|q} \sum_{\substack{e \in \Z \\ |de-a_1| \le H\\ (e,q/d)=1}} S(q; de, a_k) I\left(\bet_1 - \frac{de-a_1}{q}, \bet_k; Q\right) + O(q^{1/2 + \eps}).
\end{align}
We expect the sum on the right hand side of \eqref{3.6} to be dominated by the terms corresponding to small values of $h$. In particular, whenever $|h| \le d/2$ we have $|e-a_1/d| \le 1/2$ and hence $e=\llbracket a_1/d \rrbracket$. These terms will form our main term.
Write
\begin{align*}
	E(q, \ba; \bbet) =  q^{-1} \sum_{d|q} \sum_{\substack{e \in \Z \\ d/2 < |de-a_1| \le H}} \left|S(q; de, a_k) I\left(\alp_1 - \frac{de}{q}, \bet_k;Q\right)\right|
\end{align*}
for the sum over all the remaining terms where $h > d/2$. Then \eqref{3.6} may be rephrased as
\begin{align}\label{3.7}
	f(\alp_1, \alp_k; Q)&=  q^{-1} \sideset{}{^\dag}\sum_{d|q} S(q; d \llbracket a_1/d \rrbracket, a_k) I\left(\alp_1-\frac{\llbracket a_1/d \rrbracket}{q/d}, \bet_k; Q\right) \nonumber\\
	&\qquad + O\big(E(q, \ba; \bbet) + q^{1/2+\eps}\big).
\end{align}
Thus, it suffices to bound $E(q, \ba; \bbet)$.
By Lemma~\ref{L2.5} we see that
\begin{align}\label{3.8}
	E(q, \ba; \bbet) \ll q^{-1}(1+Q^k|\bet_k|)^{1/2} \sum_{d|q}\sum_{\substack{e \in \Z \\ d/2 < |de-a_1| \le H}}  \frac{|S(q; de, a_k)|}{|\alp_1-de/q|}.
\end{align}
Now, the condition $de \neq a_1$ together with our bound $|\bet_1| \le (2q)^{-1}$ implies that
\begin{align}\label{3.9}
	\left|\alp_1 - \frac{de}{q}\right| \ge \frac{|a_1-de|}{q} - |\bet_1| \ge \frac{|a_1-de|}{q} - \frac{1}{2q} \ge \frac{|a_1-de|}{2q}.
\end{align}
Using this within \eqref{3.8} and applying Lemma~\ref{L2.1}\eqref{L2.1b} yields the bound
\begin{align}\label{3.10}
	E(q, \ba; \bbet) &\ll (1 + Q^k|\bet_k|)^{1/2} \sum_{d|q}\sum_{\substack{e \in \Z \\ d/2 < |de-a_1| \le H}}\frac{|S(q; de, a_k)| }{|a_1-de|} \nonumber\\
	 &\ll q^{1/2+\eps}(1 + Q^k|\bet_k|)^{1/2} \sum_{d|q}d \sum_{\substack{e \in \Z \\ d/2 < |de-a_1| \le H}}|a_1-de|^{-1}  \nonumber\\
	&\ll q^{1/2+\eps}(1 + Q^k|\bet_k|)^{1/2},
\end{align}
where in the last step we used \eqref{3.5} together with standard bounds for the divisor function. The proof of \eqref{3.1} under the assumption \eqref{3.5} is now complete upon inserting  \eqref{3.10} into \eqref{3.7}, and the unconditional statement follows upon combining this with the bound \eqref{3.4}.

The second statement of Theorem~\ref{T1.1} is proved in a similar manner, and we only briefly detail the changes that need to be effected. Here, we do not need to consider a dyadic dissection of the interval, so all of our arguments will involve the exponential sum $f(\alp_1, \alp_k)$ and integral $I(\bet_1, \bet_k)$ instead of their dyadic analogues $f(\alp_1, \alp_k;Q)$ and  $I(\bet_1, \bet_k;Q)$, and will have $P$ instead of $Q$. We now observe that in the proof of Lemma~\ref{L3.1}, the condition \eqref{1.8} implies that
$$
	|\phi'(x)| = |\bet_1 - b/q + k \bet_k x^{k-1}| \le \frac{1}{2q} + \frac12 + \frac{1}{4q} \le 2.
$$
We may thus take $H_1=2$. The argument then proceeds as above, with the difference that we may skip the discussion of the case \eqref{3.3} and can continue directly with the hypothesis \eqref{3.5}. From this point we arrive, \emph{mutatis mutandis}, at  \eqref{3.7}. In order to bound the error term $E(q, \ba; \bbet)$ we now note that
\eqref{3.9} and \eqref{1.8} combine to show that, for $1 \le x \le P$, one has
$$\left|\alp_1 - \frac{de}{q} + k \bet_k x^{k-1}\right| \ge \left|\alp_1 - \frac{de}{q}\right| - \frac{1}{4q} \ge \frac{|a_1-de|}{4q}.$$ It thus follows from Lemma~\ref{L2.4}\eqref{L2.4a} that \eqref{3.8} can be replaced by
\begin{align*}
	E(q, \ba; \bbet) \ll \sum_{d|q}\sum_{\substack{e \in \Z \\ d/2 < |de-a_1| \le H}} \frac{|S(q; de, a_k)|}{|a_1-de|},
\end{align*}
and the desired bound $E(q, \ba; \bbet) \ll q^{1/2+\eps}$ follows as above. This completes the proof of Theorem~\ref{T1.1}. Moreover, Corollary \ref{C1.2}\eqref{C1.2a} is now immediate, and part \eqref{C1.2b} follows easily upon applying Theorem~7.3 of \cite{V:HL} and Lemma~\ref{L2.1}\eqref{L2.1a} within Theorem~\ref{T1.1}.

We now turn to the proof of Theorem~\ref{T1.3}. When $d \neq (a_1,q)$, the fractions $a_1/q$ and $d\llbracket a_1/d \rrbracket/q$ are distinct, and thus the latter one corresponds to a non-optimal rational approximation to $\alp_1$. In particular, we have $|\alp_1 - q^{-1}d\llbracket a_1/d \rrbracket| \ge 1/(2q)$. Thus, we can apply Lemma~\ref{L2.5} within the sum \eqref{1.9} much as above, and the statement of Theorem~\ref{T1.3}\eqref{T1.3a} follows immediately.

For the second statement of the theorem, we begin by observing  that the hypotheses imply that the sum in Theorem~\ref{T1.1} has exactly two main terms, corresponding to the values $d=1$ and $d=q$, respectively. We will focus on the latter. Note that when $\bet_k = 0$, we can explicitly compute
\begin{align*}
	I(\alp_1, 0) = \int_0^P e(\alp_1 x) \d x = \frac{e(\alp_1 P) - 1}{2 \pi i \alp_1} = e\left(\frac{\alp_1 P}{2}\right) \frac{\sin (\pi \alp_1 P)}{\pi \alp_1}.
\end{align*}
When $\|\alp_1 P \| > \del$, we have $2\del \le |\sin (\pi \alp_1 P)| \le 1$, and thus
\begin{align*}
	\frac{4q\del}{3\pi} \le \frac{2\del}{\pi |\alp_1|} \le |I(\alp_1, 0)| \le \frac{1}{\pi |\alp_1|} \le \frac{2q}{\pi}
\end{align*}
under our assumption that $1/(2q) \le \alp_1 \le 3/(2q)$. Thus, the main term corresponding to $d=q$ is given by
\begin{align*}
	q^{-1} |S(q; 0, a_k) I(\alp_1, 0)| \ge \frac{4 \del}{3 \pi}|S(q; 0, a_k)|.
\end{align*}
Finally, we note that when $q=p^k$, Lemma~4.4 in \cite{V:HL} shows that we have $|S(q; 0,a_k)| = q^{1-1/k}$, which implies the result.

\section{Additional lemmata for the proof of Theorem~\ref{T1.4}}
For the proof of our second main result, we need some more detailed information about cubic complete exponential sums.
\begin{lemma}\label{L4.1}
	Let $q \in \N$ and set
	$$
		 q_2 = \prod_{\substack{p^t\|q\\ t=1\text{ or }2}} p^t,\quad q_3=\prod_{\substack{p^t\|q\\t\ge 3}} p^t, \quad\kappa(q) =q_2^{1/2}q_3^{1/3}.
	$$
	Furthermore, suppose that $(q,a)=1$. Then
	$$
		S_{1,3}(q;b,a)\ll q^{1+\eps}\kappa(q)^{-1}.
	$$
\end{lemma}
\begin{proof}
	Suppose that $q=rs$ with $(r,s)=1$, and write $a=a_2r+a_1s$ and $b=b_2r+b_1s$. Then it follows by standard arguments that
	$$
		S_{1,3}(q;b,a)=S_{1,3}(r;b_1,a_1)S_{1,3}(s;b_2,a_2).
	$$
	We are therefore free to restrict our focus to prime power moduli. By Lemma~\ref{L2.1}\eqref{L2.1a} when $(q_3,a)=1$ we have
	$$
		S_{1,3}(q_3;b,a) \ll q_3^{2/3+\eps}.
	$$
	Thus it suffices to show that whenever $(a,p)=1$ and $t=1$ or $2$ we have
	$$
		S_{1,3}(p^t;b,a)\ll p^{t/2}.
	$$
	When $t=1$ this follows from Corollary II.2F of Schmidt \cite{WS76}.  Suppose $t=2$.  Then when $p=2$ or $3$ this bound is trivial, so we can suppose that $p>3$.  Thus
	\begin{align*}
		S_{1,3}(p^2;b,a) &= \sum_{v=0}^{p-1} \sum_{u=1}^p  e\left(\frac{b(pv+u)+a(pv+u)^3}{p^2}	\right) \\
		&= \sum_{u=1}^p e\left( \frac{bu+au^3}{p^2} \right) \sum_{v=0}^{p-1} \left(	\frac{b+3au^2}{p}v 	\right)\\
		& = p\sum_{\substack{u=1\\	b+3au^2\equiv 0 \mmod p}}^p e\left( \frac{bu+au^3}{p^2} 	\right).
	\end{align*}
	The congruence $b+3au^2\equiv 0 \mmod p$ has at most two solutions, and it follows that $|S_{1,3}(p^2;b,a)| \le 2p$. The claim of the lemma follows upon collecting our results.
\end{proof}

\begin{lemma}
	\label{L4.2}
	Let $q\in\mathbb N$ be odd and $c\in\mathbb Z$ with $(q,c)=1$.  Then there exists $a\in\mathbb Z$ such that $(q,a)=1$ and for every $\eps>0$ we have
	$$
		|S_{1,3}(q;a-c,a)|\gg q^{1/2-\eps}.
	$$
\end{lemma}
\begin{proof}
	Before embarking on the argument, we note that the lemma is trivial for $q=1$, and hence we may assume $q>1$ and consequently $c \neq 0$.
	Again we use the multiplicative property of the Gauss sum as described at the start of the proof of the previous lemma. Thus it suffices to establish the lemma for prime powers.
	\par

	Let $p$ be an odd prime, $t\in\mathbb N$ and $c\in\mathbb Z$ be such that $p\nmid c$. The lemma follows if we are able to show that there is an absolute constant $\xi>0$ having the property that
	\begin{align*}
		|S_{1,3}(p^t;a-c,a)|\ge \xi p^{t/2}
	\end{align*}
	for all odd prime powers $p^t$.
	
	First we deal with the case $t=1$. Clearly, the desired statement follows if we can show that
	\begin{align}\label{4.1}
		\sum_{a=1}^{p-1}|S_{1,3}(p;a-c,a)|^2\ge \xi p^2.
	\end{align}
	In general, we have
	\begin{align*}
		\sum_{a=1}^{p-1} |S_{1,3}(p;a-c,a)|^2 =
		p \sum_{\substack{m,n=1\\n+n^3\equiv m+m^3\mmod p}}^{p} e\left(\frac{c(n-m)}{p}\right) - \left|\sum_{m=1}^{p} e\left(\frac{cm}{p}\right)\right|^2.
	\end{align*}
	Clearly, the second sum vanishes.  Hence
	\begin{align}\label{4.2}
	\sum_{a=1}^{p-1} |S_{1,3}(p;a-c,a)|^2
	&= p \sum_{m=1}^{p} \sum_{\substack{n=1\\n+n^3\equiv m+m^3\mmod p}}^{p} e\left(\frac{c(n-m)}{p}\right) \nonumber\\
	& =p^2 + p\sum_{m=1}^p \sum_{\substack{h=1\\	3m^2+3hm+h^2+1\equiv 0\mmod p}}^{p-1} e\left(	\frac{ch}p 	\right).
	\end{align}
	When $p=3$ the congruence in the inner sum becomes $h^2\equiv -1\pmod 3$ which is insoluble, so that sum vanishes, and \eqref{4.2} reads
	$$
		\sum_{a=1}^{2} |S_{1,3}(3;a-c,a)|^2\\	= 3^2.
	$$
	\par
	
	When $p>3$, on the other hand, we use that
	$$
		12(3m^2+3hm+h^2+1)=  	(6m+3h)^2+3h^2+12,
	$$
	whence we obtain
	$$
		p\sum_{m=1}^p \sum_{\substack{h=1\\	3m^2+3hm+h^2+1\equiv 0\mmod p}}^{p-1} e\left(	\frac{ch}p 	\right) = p\sum_{h=1}^{p-1} e\left(\frac{ch}p	\right)\left(1+\left(\frac{-3h^2-12}{p}	\right)_L\right),
	$$
	where $\big(\frac{a}{p}\big)_L$ denotes the Legendre symbol.
	Thus, upon extending the sum on the right hand side to include the term $h=p$ also, while noting that $-3p^2-12 \equiv 2^2 (-3) \mmod p$, we obtain
	\begin{align}\label{4.3} 
		\sum_{a=1}^{p-1} |S_{1,3}(p;a-c,a)|^2 = p^2 &- p\left(1+\left(	\frac{-3}{p}\right)_L \right) \notag \\
		&+ p\sum_{h=1}^{p} e\left(\frac{ch}p \right)\left( \frac{-3h^2-12}{p} \right)_L.
	\end{align}
	At this point, we see from Theorem~II.2G in \cite{WS76} that
	\begin{align}\label{4.4}
		\sum_{h=1}^{p} e\left(\frac{ch}p \right)\left( \frac{-3h^2-12}{p} \right)_L \le 2 p^{1/2},
	\end{align}
	whence we find that
	\begin{align*}
		\sum_{a=1}^{p-1} |S_{1,3}(p;a-c,a)|^2\ge p^2-2p-2p^{3/2}.
	\end{align*}
	When $p>7$, this is already sufficient for \eqref{4.1}, so it remains to analyse the cases when $p=5$ and $p=7$.\par
	
	Let now $p=5$. Since $\left(\frac{-3}{5}\right)_L=-1$, the desired bound \eqref{4.1} follows with $\xi = 1 - 2/\sqrt 5$ upon deploying \eqref{4.4} within the expression in \eqref{4.3}.
	Finally, consider the case $p=7$. Since $\left(\frac{-3}{7}\right)_L=1$, we have in \eqref{4.3} that
	$$
		\sum_{a=1}^{6} |S_{1,3}(7;a-c,a)|^2 = 42 + 7\sum_{h=1}^{6} e\left(\frac{ch}7\right)\left(\frac{-3h^2-12}{7}	\right)_L,
	$$
	where we observed that the summand corresponding to $h=7$ is $1$.
	The remaining sum over $h$ is
	\begin{align*}
		&-e\left(\frac{c}{7}\right)+e\left(	\frac{2c}{7}\right)-e\left(	\frac{3c}{7}	\right)-e\left( \frac{4c}{7} \right)+e\left(\frac{5c}{7}\right)-e\left(\frac{6c}{7}	\right),
	\end{align*}
	which has absolute value smaller than $6$ for all $c \in \{1, \ldots, 6\}$.  It follows that \eqref{4.1} holds in this case also.
	Thus whenever $t=1$ there is at least one $a$ satisfying the necessary requirements.
	\par
	
	Now consider the case $t\ge 2$.  Suppose first that $p>3$.  Write $t=3v+u$ with $v\ge0$ and $1\le u\le 3$.  When $u\not=1$ choose $a=c$.  Then by iteratively applying Lemma~4.4 of \cite{V:HL} we find that
	$$
		S_{1,3}(p^t;a-c,a) = S_{1,3}(p^t;0,a)= p^{2v+u-1}\ge p^{t/2},
	$$
	since in the notation of that lemma and (2.25) {\it ibidem} we have $l=t>1=\gamma$.
	\par
	
	When $u=1$, we may now assume that $v\ge 1$. Choose $a$ so that $a\equiv c+p^{2v}(a'-c)\pmod{p^t}$ where $a'$ is at our disposal.  Put $m=p^{3v}x+y$.  Then our sum is
	\begin{align*}
		S_{1,3}(p^{3v+1};p^{2v}(a'-c),a) = \sum_{y=1}^{p^{3v}} \sum_{x=0}^{p-1} e\left(\frac{3ay^2x}{p} + \frac{(a'-c)y}{p^{v+1}} + \frac{ay^3}{p^{3v+1}}\right).
	\end{align*}
	The sum over $x$ is $0$ unless $p|y$ in which case it sums to $p$, and thus the above is
	\begin{align*}
		p\sum_{z=1}^{p^{3v-1}} e\left(\frac{(a'-c)z}{p^{v}} + \frac{az^3}{p^{3(v-1)+1}}	\right) = p^2S_{1,3}(p^{3(v-1)+1}; p^{2v-2}(a'-c), a).
	\end{align*}
	Iterating this argument gives
	$$
		S_{1,3}(p^t;a-c,a) = p^{2v}S_{1,3}(p;a'-c,a).
	$$
	Now we choose $a'$ in accordance with the case $t=1$ above.  Thus
	$$
		|S_{1,3}(p^t;a-c,a)|\ge \xi p^{2v+1/2}\ge \xi p^{t/2}.
	$$
	\par
	When $p=3$ we can apply a slightly modified argument.  Now in the notation (2.25) of  \cite{V:HL} we have $\gamma=2$.  When $t=3v+u$ with $u=2$ or $3$ we again take $a=c$ and obtain, by Lemma~4.4 {\it ibidem},
	$$
		S_{1,3}(3^t;a-c,a) = 3^{2v}S_{1,3}(3^u;0,a)
	$$
	and this is $3^{2v+2}$ when $u=3$.  When $u=2$, we have instead
	$$
		S_{1,3}(3^u;0,a) = 3+6\cos(2 \pi a/9).
	$$
	Since $a$ is not divisible by $3$, the cosine cannot be $-\frac12$.  Thus
	$$	
		|S_{1,3}(3^t;a-c,a)|\ge \xi 3^{t/2}.
	$$
	When $u=1$ we follow the recipe for general $p$ and obtain
	$$
		S_{1,3}(3^t;a-c,a) = 3^{2v}S_{1,3}(3;a'-c,a)
	$$
	and again appeal to the case $t=1$.
\end{proof}

\section{Prolegomena to the proof of Theorem~\ref{T1.4}}
\label{S5}

\noindent Before proceeding to the various parts of the proof of Theorem~\ref{T1.4}, it is useful to review some measure theoretic aspects of approximation of real numbers by rational numbers. In view of the periodicity of our functions, we concentrate on the interval $[0,1]$.

As is well known, Dirichlet's approximation theorem states that every real number $\gamma$ has the property that there are arbitrarily large $q\in\mathbb N$ and $c\in\mathbb Z$ such that $(q,c)=1$ and $|\gamma-c/q|\le q^{-2}$. Moreover, it follows from Khinchine's theorem (see e.g. Theorem II.2B in \cite{WS80}) that for almost every such $\gam$ there is a positive number $C(\gamma)$ such that whenever $(q,c)=1$ we have
\begin{align}\label{5.1}
	\frac{C(\gamma)}{q^{2}(\log 2q)^{2}}\le \left| \gamma - \frac{c}{q}\right|.
\end{align}
In particular there is a subset $\Gamma$ of $(\mathbb R\setminus\mathbb Q)\cap[0,1]$ having these properties and with $\mes \Gamma =1$.

\par

For the upper bound when $k=3$ we need to refine this further.  Let $\Gamma_0$ denote the subset of $\Gamma$ with the property that for every $\delta>0$ and $\gamma\in\Gamma_0$ there are, in the notation of Lemma~\ref{L4.1}, at most a finite number of $q$ and $c$ with
\begin{align}\label{5.2}
	\left| \gamma-\frac{c}{q} \right|\le q_2^{-2-\delta}q_3^{-4/3-\delta}.
\end{align}
\begin{lemma}\label{L5.1}
	The set $\Gamma_0$ has full measure in $[0,1]$.
\end{lemma}

\begin{proof}
	Let $\Upsilon_0=\bigcup_{\delta>0}\Upsilon(\delta)$ where $\Upsilon(\delta)$ denotes the set of $\gamma \in [0,1]$ having the property that \eqref{5.2} holds for infinitely many pairs $(q,c)$.
	Let further $N\in\mathbb N$.  Then
	$$
		\Upsilon(\delta) \subseteq \bigcup_{q>N} \bigcup_{0\le c\le q} \left\{ \gamma : \left|\gamma-\frac{c}{q} \right|\le q_2^{-2-\delta}q_3^{-4/3-\delta} \right\}.
	$$
	Hence
	$$
		\mes \Upsilon(\delta) \le \sum_{q_2q_3>N}4q_2^{-1-\delta}q_3^{-1/3-\delta}.
	$$
	We have
	$$
		\sum_{q_3\le X}1 \le \sum_{\substack{r^3l\le X\\ l|r^2}} 1 \le \sum_{r\le X^{1/3}} d(r^2) \ll X^{1/3+\eps}.
	$$
	Therefore for any $Y>0$ we have
	\begin{align*}
		\sum_{q_3>Y} q_3^{-1/3-\delta} = \sum_{k=0}^{\infty} \sum_{2^kY<q_3\le 2^{k+1}Y} q_3^{-1/3-\delta} \ll Y^{-\delta/2},
	\end{align*}
	and so
	\begin{align*}
		\sum_{q_2q_3>N} 4q_2^{-1-\delta}q_3^{-1/3-\delta} &\ll \sum_{q_2\le N}  q_2^{-1-\delta} \sum_{q_3>N/q_2} q_3^{-1/3-\delta} + \sum_{q_2> N}  4q_2^{-1-\delta} \\
		&\ll \sum_{q_2\le N}  q_2^{-1-\del/2}N^{-\delta/2} + N^{-\delta}.
	\end{align*}
	Thus
	$$
		\mes \Upsilon(\delta) \ll N^{-\delta/2}
	$$
	and this holds for every $N\in\mathbb N$.
\end{proof}

\section{Theorem~\ref{T1.4} : The upper bound when $k=2$}\label{S6}

\noindent Let $\alpha\in \mathbb R$ and $\gam \in \Gam$.  By Dirichlet's theorem on diophantine approximation we can choose $a_2$, $q$ with $(a_2,q)=1$, $q\le Q$ and
$$
	\left|\alpha+\gamma-\frac{a_2}{q}\right|\le \frac1{q Q}.
$$
Then choose $a_1$ so that
$$
	\left|\alpha-\frac{a_1}{q}\right|\le\frac1{2q}.
$$
Set $\bet_1 = \alp-a_1/q$ and $\bet_2 = \alp+\gam-a_2/q$. Hence, by Theorem~8 of \cite{V:09}, we have
$$
	f_{1,2}(\alpha,\alpha+\gamma;Q) = q^{-1}S_{1,2}(q;a_1,a_2) I_{1,2}(\beta_1,\beta_2;Q) + O\big(Q^{1/2} \big).
$$
Since $(a_2,q)=1$ we have $|S_{1,2}(q;a_1,a_2)| \ll \sqrt{q}$ by classical bounds on the Gauss sum. Thus
$$
	f_{1,2}(\alpha,\alpha+\gamma;Q) \ll |I_{1,2}(\beta_1, \beta_2;Q)| q^{-1/2} + Q^{1/2}.
$$
Therefore, by Theorem~7.1 in \cite{V:HL} we have the bound
\begin{align}\label{6.1}
	f_{1,2}(\alpha,\alpha+\gamma;Q) \ll \frac{Q}{(q+q Q|\beta_1|+q Q^2|\beta_2|)^{1/2}} + Q^{1/2}.
\end{align}
We may assume that
\begin{align}\label{6.2}
	q\le \big(\textstyle{\frac12}C(\eps,\gamma)Q\big)^{1/(2+\eps)},
\end{align}
for otherwise in the denominator in \eqref{6.1} we have
$$
	q+q Q|\beta_1|+q Q^2|\beta_2| \ge\big(\textstyle{\frac12}C(\eps,\gamma)Q\big)^{1/(2+\eps)}
$$
trivially.
Since $\gam \in \Gam$, we infer from (\ref{5.1}) via \eqref{6.2} that
$$
	|\beta_1|\ge \left|\gamma-\frac{a_2-a_1}{q}\right|-|\beta_2| \ge \frac{C(\eps,\gamma)}{q^{2+\eps}} - \frac{1}{q Q}\ge \frac12C(\eps,\gamma) q^{-2-\eps},
$$
provided that $Q\ge 2/C(\eps,\gamma)$, which we may certainly assume. Thus, we find
$$
	q+q Q|\beta_1| \ge q+{\textstyle\frac12}Q C(\eps, \gam)q^{-1-\eps} \ge ({\textstyle\frac12}C(\eps,\gamma) Q)^{1/(2+\eps)}
$$
in this case as well, and \eqref{6.1} becomes
$$
	f_{1,2}(\alpha,\alpha+\gamma;Q) \ll_{\eps,\gamma} Q^{1-1/(4+2\eps)}.
$$
We conclude that if $\theta>\frac34$, then
$$
	Q^{-\theta} \sup_{\alpha\in[0,1)} |f_{1,2}(\alpha,\alpha+\gamma;Q)| \ll_{\theta,\gamma} 1
$$
as required.

\section{Theorem~\ref{T1.4} : The upper bound when $k=3$}

\noindent This follows the pattern set in \S\ref{S6}.  Again we use (\ref{5.1}), but now we suppose that $\gamma\in\Gamma_0$, and hence given $\gamma$ we may suppose that for any fixed $\delta>0$  the inequality (\ref{5.2}) holds for at most a finite number of $q$ and $c$.  It will be convenient to also suppose that $\delta$ is sufficiently small.  For such a $\gamma$ we show that for arbitrarily large $Q$ we have
\begin{align}\label{7.1}
	f_{1,3}(\alpha,\alpha+\gamma;Q) \ll_{\delta,\gamma} Q^{3/4+\delta}
\end{align}
uniformly for $\alpha\in\mathbb R$.\par

Let $\alpha\in \mathbb R$.  By Dirichlet's theorem on diophantine approximation we can choose $a_3$, $q$ with
\begin{align}\label{7.2}
	(a_3,q)= 1, \qquad  q\le Q^{3/2} \qquad \text{ and } \qquad  \left|\alpha+\gamma-\frac{a_3}{q}\right|\le \frac1{q Q^{3/2}}.
\end{align}
Then choose $a_1$ so that
\begin{equation*}
	-\frac1{2q}<\alpha-\frac{a_1}{q} \le\frac1{2q},
\end{equation*}
and let $\beta_3=\alpha+\gamma-a_3/q$ and $\beta_1=\alpha-a_1/q$.  
By Corollary \ref{C1.2}\eqref{C1.2a} we have
\begin{align}\label{7.3}
	f_{1,3}(\alpha,\alpha+\gamma;Q) &= q^{-1} \sideset{}{^\dag}\sum_{d|q} S_{1,3}(q;d\llbracket a_1/d\rrbracket,a_3)I_{1,3}\left(\alpha-\frac{d\llbracket a_1/d \rrbracket}{q},\beta_3; Q\right)\nonumber\\
	&\qquad	+ O\big(Q^{3/4+\eps}\big).
\end{align}

Since $(a_3,q)=1$, by Theorem~7.3 of \cite{V:HL} together with Lemma~\ref{L4.1} we have
\begin{align*}
	q^{-1}S_{1,3}(q;d\llbracket &a_1/d\rrbracket,a_3)I_{1,3}\left(\alpha-\frac{d\llbracket a_1/d \rrbracket}{q},\beta_3; Q\right) \\
	&\ll Q^{1+\eps} \kappa(q)^{-1} \left(1+Q\left|\alpha-\frac{d\llbracket a_1/d \rrbracket}{q}\right|+Q^3|\beta_3|\right)^{-1/3}.
\end{align*}
It follows that the contribution arising from those $d|q$ for which
\begin{align}\label{7.4}
	\kappa(q)^3\left(1+Q\left|\alpha-\frac{d\llbracket a_1/d \rrbracket}{q}\right|+Q^3|\beta_3|\right) \ge Q^{3/4-2\delta}
\end{align}
is satisfactory in view of \eqref{7.1}.

Thus it remains to deal with any possible terms for which \eqref{7.4} fails to hold.
Suppose that $d \neq (a_1,q)$ for some $d$ violating \eqref{7.4}, so that
$$
	\left|\alpha-\frac{d\llbracket a_1/d \rrbracket}{q}\right| \ge\frac1{2q}.
$$
In that case we would have
$$
	q_2^{1/2}Q = q_2^{3/2}q_3q^{-1}Q \le 2\kappa(q)^3 Q\left|\alpha-\frac{d\llbracket a_1/d \rrbracket}{q}\right|\le 2 Q^{3/4-2\delta},
$$
which is impossible for $Q$ large. Hence the only term in \eqref{7.3} that could possibly violate \eqref{7.4} is the one corresponding to $d=(q, a_1)$, for which
$$
	\alpha-\frac{d\llbracket a_1/d \rrbracket}{q}= \alp - \frac{a_1}{q}=\beta_1.
$$
For this term the negation of \eqref{7.4} reads
\begin{align}\label{7.5}
	\kappa(q)^3\big(1+Q|\beta_1|+Q^3|\beta_3|\big) < Q^{3/4-2\delta},
\end{align}
and we observe that in this instance
\begin{align}\label{7.6}
	q \leq q_2^{3/2}q_3 = \kappa(q)^3 \le Q^{3/4-2\delta}.
\end{align}

We assume there is such a term and show that it contradicts the assumption $\gamma\in\Gamma_0$. Since $\Gam_0 \subseteq \Gam$, we see from  (\ref{5.1}) and \eqref{7.2} that
$$
	\frac{C(\gamma)}{q^{2}(\log 2q)^2} <\left|\gamma-\frac{a_3-a_1}{q}\right|\le |\beta_1|+|\beta_3|\le |\beta_1| + q^{-1}Q^{-3/2}.
$$
Since by \eqref{7.6} we may suppose that $Q$ is large enough so that
$$
	q^{-1}Q^{-3/2}\le \frac{C(\gamma)}{2q^{2}(\log 2q)^2}\le \frac12\left|\gamma-\frac{a_3-a_1}{q}\right|,
$$
it follows from our assumption \eqref{7.5} that
\begin{align}\label{7.7}
	\left|\gamma-\frac{a_3-a_1}{q}\right| \le 2|\beta_1|\le 2\kappa(q)^{-3}Q^{-1/4-2\delta}.
\end{align}
We also have
$$
	\frac{C(\gamma)}{q^{2}(\log 2q)^2}\le 2\kappa(q)^{-3}Q^{-1/4-2\delta}
$$
and therefore
$$
	C(\gamma)Q^{1/4+2\delta}\le 2q_2^{1/2}q_3(\log 2q)^2 \le 2q(\log 2q)^2.
$$
Thus, as $Q$ can be taken large enough so that $(\log 2q)^2 < \frac{1}{2}C(\gamma)Q^{2\delta}$, we have $q>Q^{1/4}$. By (\ref{7.6}) we have
\begin{align*}
	Q^{-1/4-2\delta} &= \big(Q^{-3/4+2\delta} \big)^{1/3+\frac{32\delta}{9-24\delta}}\\
	&\le \big(q_2^{-3/2}q_3^{-1} \big)^{1/3+\frac{32\delta}{9-24\delta}}\\
	&\le q_2^{-1/2-2\delta}q_3^{-1/3-2\delta}\\
	&\le \textstyle{\frac12} q_2^{-1/2-\delta}q_3^{-1/3-\delta}.
\end{align*}
Hence, by (\ref{7.7}) we find that
$$
	\left|\gamma-\frac{a_3-a_1}{q}\right|\le q_2^{-2-\delta}q_3^{-4/3-\delta}.
$$
However, by the definition of $\Gamma_0$ in Section \ref{S5} this is expressly excluded for large $q$, and so this establishes as promised that (\ref{7.5}) is impossible. This completes the proof of (\ref{7.1}), which gives the conclusion $\theta_3\le \frac34$.

\section{Theorem~\ref{T1.4}: The lower bound}\label{S8}

\noindent Let $\delta>0$ be sufficiently small, and let $k = 2$ or $3$. We show that for all $\gamma\in\mathbb R\setminus \Q$ there are arbitrarily large $Q$ such that
$$
	\sup_{\alpha\in [0,1)} \left|\sum_{Q<n\le 2 Q} e\big(\alpha n + (\alpha +\gamma)n^k\big)\right| \gg Q^{3/4-\delta}.
$$
The continued fraction algorithm for $\gamma$ gives $q$ and $c$ with $q$ arbitrarily large, $(q,c)=1$, and
$$
	|\gamma-c/q|\le q^{-2}.
$$
Note that two successive convergents $c/q$ and $c'/q'$ of the continued fraction satisfy $cq'-c'q=\pm 1$ and so $(q,q')=1$.  Thus either $q$ or $q'$ is odd. For an arbitrary odd convergent $q$ and a fixed small parameter $\delta>0$ set
\begin{align}\label{8.1}
	Q=q^{2/(1+2\delta)}.
\end{align}

Let $a_k$ be any integer with $(q,a_k)=1$, and if $k=3$ assume additionally that $a_3$ is such that $S_{1,3}(q; a_3-c,a_3) \gg q^{1/2 - \del}$. The existence of such an $a_3$ is guaranteed by Lemma~\ref{L4.2}. Take further $\alpha = -\gamma+a_k/q$ and $a_1=a_k-c$, and define $\alpha_k = \alpha+\gamma$, $\alpha_1=\alpha$ and $\beta_j=\alpha_j-a_j/q$ for $j = 1, k$. Thus
\begin{align*}
	\beta_k=\alpha+\gamma-\frac{a_k}{q}=0 \qquad \text{and} \qquad \beta_1=-\gamma+ \frac{c}{q},
\end{align*}
and so in particular
$$
	|\beta_1|\le q^{-2}\le Q^{-1-2\delta}.
$$
We have
\begin{align}\label{8.2}
	f_{1,k}(\alpha,\alpha+\gamma; Q) = q^{-1} S_{1,k}(q;a_1,a_k) I_{1,k}(-\gam+c/q,0; Q) + O\left(Q^{1/3 + \del}\right).
\end{align}
When $k=2$, this is Theorem~8 in \cite{V:09}, and for $k=3$ it is a consequence of the second statement of Theorem~\ref{T1.1}. Indeed, if $d \neq (a_1, q)$ so that $d \llbracket a_1/d\rrbracket \neq a_1$, then $|\alpha-d\llbracket a_1/d \rrbracket/q|>1/(2q)$ and so
$$
	\left|I_{1,3}\left(\alpha-\frac{d\llbracket a_1/d \rrbracket}{q},0;Q\right)\right|\ll q
$$
by Lemma~\ref{L2.4}\eqref{L2.4a}. Moreover, by Lemma~\ref{L2.1}\eqref{L2.1a} we have
$$
	S_{1,3}(q;d\llbracket a_1/d\rrbracket,a_3)\ll q^{2/3+\eps}.
$$
Hence we see from \eqref{8.1} that
\begin{align*}
	\sideset{}{^\dag}\sum_{\substack{d|q\\d \neq (a_1,q)}} S_{1,3}(q;d\llbracket a_1/d\rrbracket,a_3)I_{1,3}\left(\alpha-\frac{d\llbracket a_1/d \rrbracket}{q},0;Q\right) \ll \sum_{d|q} q^{2/3+\eps} \ll Q^{1/3+\delta/2},
\end{align*}
and consequently \eqref{3.1} implies 
 that
\begin{align*}
	&f_{1,3}(\alpha,\alpha+\gamma;Q) -  q^{-1} S_{1,3}(q;a_1,a_3) I_{1,3}(-\gam+c/q,0; Q)\\
	&\qquad = \frac1q \sideset{}{^\dag}\sum_{\substack{d|q \\ d \neq (a_1,q)}} S_{1,3}(q;d\llbracket a_1/d\rrbracket,a_3)I_{1,3}\left(\alpha-\frac{d\llbracket a_1/d \rrbracket}{q},0;Q\right) + O(q^{1/2+\eps})\\
	&\qquad \ll Q^{1/3 + \del/2} + q^{1/2+\eps} \ll Q^{1/3+\del}
\end{align*}
as claimed.

Note that
\begin{align}\label{8.3}
	I_{1,k}(\beta_1,0; Q) = \int_{Q}^{2 Q} e(\beta_1x) \d x = Q e(3 Q\beta_1/2) \frac{\sin(\pi\beta_1Q)}{\pi\beta_1 Q} \gg Q,
\end{align}
where in the last step we used that $Q|\beta_1| =Q|\gamma-c/q |\le Q^{-2\delta}$ and thus
$$
	\frac{\sin(\pi\beta_1 Q)}{\pi\beta_1 Q} \gg 1.
$$
Moreover, since $q$ is odd and $(a_k,q)=1$, we have
\begin{align}\label{8.4}
	q^{-1}|S_{1,k}(q;a_1,a_k)| \gg q^{-1/2-\del} \gg Q^{-1/4-2\del}.
\end{align}
This follows from the classical bound for quadratic Gauss sums in the case $k=2$, and is a consequence of our choice of $a_3$ via Lemma~\ref{L4.2} when $k=3$.
Hence upon combining \eqref{8.2}, \eqref{8.3} and \eqref{8.4}, we see that
$$
	|f_{1,k}(\alpha,\alpha+\gamma; Q)| \gg Q^{3/4-2\delta},
$$
and the theorem follows.

\section{Concluding Remarks}

\noindent In conclusion, we make some remarks about what might be the real truth with regard to estimates for Weyl sums on suitable sets of minor arcs. Consider sums of the kind \eqref{*1.1} with multidegree $\bk = (1, \ldots, k)$, where $\balp$ is `sufficiently random'.
Then we might guess that in this circumstance $f_{\bk}(\balp)$ behaves like a sum of independent random unimodular variables and so the central limit theorem suggests that
\begin{align}\label{eq:C3}
	f_{\bk}(\boldsymbol\alpha) \ll P^{1/2+\varepsilon}.
\end{align}
In fact, the received wisdom states that for \eqref{eq:C3} to be true, it should suffice that $\boldsymbol\alpha$ is such that for any $a_1,\ldots,a_k,q$ satisfying \eqref{1.4} with $(q,a_1,\ldots,a_k)=1$ we have
\begin{align}\label{eq:C2}
	q+qP|\alpha_1-a_1/q|+\ldots+ qP^k|\alpha_k-a_k/q|>P.
\end{align}

In the one-dimensional case, $\bk = k$, the last author and Wooley \cite{VW} show that there is a connection between the size of $f_k(\alp)$ and the number of solutions of
$$
	x_1^k+\ldots+x_b^k=y_1^k+\ldots+y_b^k
$$
with $0<x_j,y_j\le P$, and that the values of $|f_k(\alp)|$ will even have a normal distribution on the minor arcs.
For $k=2$ we have more precise bounds (see \cite[Theorem~7)]{V:09}) since we can cover all of $[0,1]^2$ by choosing $q,a_2$ with $(q,a_2)=1$, $|\alpha_2-a_2/q| \le 1/(5qP)$, $q\le 5P$ and then taking $a_1$ with $|\alpha_1-a_1/q|\le 1/(2q)$.  Thus the main interest lies in the case $k\ge 3$.  In order to make a proper comparison with the current state of play we first review what is known.\par

Take $\bk = (1, \ldots, k)$. The Vinogradov mean value theorem (Bourgain, Demeter, Guth \cite{BDG} and Wooley \cite{W:mvmt3,W:nec}) combined with Theorem~5.2 of Vaughan \cite{V:HL} give
\begin{align}\label{eq:C4}
	f_{\bk}(\balp) \ll P^{1+\eps} \left(\frac1q_j+\frac1P+\frac{q_j}{P^j}	\right)^{\textstyle\frac1{k(k-1)}}
\end{align}
when for some $j \ge 2$ there are coprime $q_j$ and $a_j$ such that $|\alpha_j-a_j/q_j|\le q_j^{-2}$.  At least when $q_j+P^j|\alpha_j-a_j| >P^2$ for every $j$ and choice of $q_j$, $a_j$, this is the best that we know when $k\ge 6$, and is perhaps also the best that we know when $3\le k\le 5$ and we need to approximate some (presumably non-zero) $\alpha_j$ with $2\le j\le k-1$.  When we have $q$, $a_k$ with $(q,a_k)=1$ and $|\alpha_k-a_k/q| \le q^{-2}$, Weyl's inequality \cite[Lemma~2.4]{V:HL} gives
\begin{align}\label{eq:C6}
	f_{\bk}(\boldsymbol\alpha) \ll P^{1+\varepsilon} \left(\frac1q+\frac1P+\frac{q}{P^k}\right)^{2^{1-k}}
\end{align}
and this is superior to (\ref{eq:C4}) when $3\le k\le 5$.  \par

There is an underlying problem when dealing with a sum as general as \eqref{*1.1}.  It seems that one ought to consider a general rational approximation to $\boldsymbol\alpha$ of the kind $|q\alpha_j-a_j|\le Q_j^{-1}$ with $(q,a_1,\ldots,a_k)=1$, but then to cover the whole of $[0,1]^k$ one needs that $q$ can be as large as $cQ_1\cdots Q_k$. This means that either $q$ might be much larger than $P^k$ or the intervals about each rational are too large to be able to deduce anything useful.  The alternative, as in the statement of (\ref{eq:C4}) above, is to deal with one $j$ at a time.  If $q_j>P^{\theta_j}$ for some $\theta_j\le 1$ then we are done and that leaves the situation when $q_j\le P^{\theta_j}$ for every $j$.  Now one can take $q= \lcm(q_1,\ldots,q_k)$ and the rational approximation to $\boldsymbol\alpha$ becomes $(a_1q/q_1,\ldots,a_kq/q_k)$.  However, the likelihood is that any useful bound will still need $q$ fairly constrained in terms of $P$ and so the $\theta_j$ will have to be rather small.  An example of this process is given in the proof of  \cite[Theorem~7.4]{V:HL}.  Some aspects of methods to overcome this are described in Chapter 5 of Baker \cite{RB86}.

Whilst our result of Theorem~\ref{T1.4} does not strictly contradict such heuristics as detailed in the opening paragraph of this section, it does raise the question of whether these heuristics might not be too naive in some cases. When $\{p_1, \ldots, p_t\}$ is a set of polynomials with a non-vanishing Wronskian, consider the associated exponential sum
\begin{align}\label{eq.1}
	f_{\bp}(\alp_1, \ldots, \alp_t) = \sum_{1 \le n \le P} e(\alp_1 p_1(n) + \ldots + \alp_t p_t(n)).
\end{align}
We know from standard bounds that $\sup_{\balp} |f_{\bp}(\balp)| = f_{\bp}(\boldsymbol 0) = \lfloor P \rfloor$, whilst on the other hand it follows from \cite[Corollary~2.2]{ChS} that whenever the polynomials $p_j$ have a non-vanishing Wronskian, the bound \eqref{eq:C3} holds for a set of $\balp \in [0,1]^t$ of full measure. Meanwhile, the analogous inequality to  \eqref{eq:C2} is not sufficient for the bound \eqref{eq:C3} to hold. This is evidenced in our Theorem \ref{T1.4} with the choices $p_1(n)=n^k$ and $p_2(n)=n^k+n$, where $k=2$ or $k=3$. Here, it transpires from our arguments that even if $\alp_1$ lacks a good rational approximation, for certain choices of $\alp_2$ the contributions of the degree $k$ part of the polynomial more or less cancel out, leading to a less random behaviour.
It is not clear to the authors whether this behaviour is particular to the presence and influence of the linear term in $p_2$ or whether there is some underlying phenomenon at work.
In the latter scenario, one might now be inclined to guess that if only $r$ of the $t$ coefficients are restricted to lie in a set of full measure whilst the other $t-r$ coefficients are allowed to range over the entire unit interval, that the ensuing bound would interpolate between the two extremes. Thus, one might speculate that when the polynomials $p_j$ are all of the same degree one has the bound
\begin{align*}
	\sup_{\alp_1, \ldots, \alp_r} |f_{\bp}(\balp)| \ll P^{1-r/(2t) + \eps} \qquad \text{for almost all $\alp_{r+1}, \ldots, \alp_t$},
\end{align*}
and that this might in some cases even be sharp (up to epsilon) for a sequence of values $P$ tending to infinity. The exponent here is $1-r/(2t) = (r/2 + (t-r))/t$ and interpolates between $r$ contributions of $1/2$ and $t-r$ contributions of $1$. Whilst this is compatible with the bounds of our Theorem \ref{T1.4}, there is still hope that stronger bounds may be available if the polynomials in question differ by more than a linear term.

Our understanding is better in the one-dimensional case. On page 43 of Vaughan \cite{V:HL} it was stated (we have changed the notation to be consistent with this memoir) that when $(q,a)=1$ and $\beta=\alpha-a/q$ it would be very interesting to decide whether the relation
\begin{align*}
	f_k(\alpha) = q^{-1}S_k(q,a) I_k(\bet) + O\left((q+qP^k|\beta|)^{\theta}\right)
\end{align*}
holds with an exponent $\theta$ smaller than $1/2$, and it was even speculated that $\theta$ might be as small as $1/k$.  This was shown to be false by Daemen \cite{Dae} and Br\"udern and Daemen \cite{BD}. One other result has come to our attention.  Heath-Brown \cite{HB10} has shown on the assumption of the $abc$ conjecture that if $\alpha$ is a quadratic irrational then
\begin{align*}
	\sum_{n\le P} e(\alpha n^3) \ll P^{\frac57+\varepsilon}.
\end{align*}
It may be worthwhile to note that quadratic irrationals are badly approximable numbers so that Heath-Brown's result can be viewed as applying to an `extreme minor arc' situation.
\par

Finally, we briefly outline an  argument, versions of which in other contexts are quite well known, that shows that one cannot expect to bound the exponential sum $f_k(\alpha)$ by anything smaller than $P^{1/2}$ on the minor arcs. Let $P$ be large and choose $R=P^{1+\phi}$ and $Q=P^{k-1-\psi}$, where $\phi$ and $\psi$ are positive numbers at our disposal and $\phi<\psi$ so that $RQ<P^{k-\delta}$ for some substantial $\delta>0$.  There are various wrinkles that could be introduced to enable a quite large $\delta$. 
\par

Let $\mathfrak M$ denote the union of the intervals $[a/q-1/(qQ),a/q+1/(qQ)]$ with $1\le a\le q\le R$ and $(q,a)=1$ and let $\mathfrak m = (1/Q,1+1/Q]\setminus\mathfrak M$.  The total measure of $\mathfrak M$ is $\le 2RQ^{-1}$ and so
$$
	\int_{\mathfrak M} |f_k(\alpha)|^2 \d\alpha \ll P^2RQ^{-1} \ll P^{4+\phi+\psi-k}.
$$
When $k\ge 4$ this is $\ll P^{1-\delta}$ for some $\delta>0$ as long as $\phi+\psi<1-\delta$.  When $k=3$ this argument can be refined by using the approximations for $f_k(\alp)$ given by (4.13), Theorem~4.2 and the integral version of Lemma~2.8 of Vaughan \cite{V:HL}.  Then we obtain the bound
\begin{align*}
\int_{\mathfrak M} |f_k(\alpha)|^2 \d\alpha &\ll	\sum_{q\le R} q \int_0^{1/(qQ)} \left(\frac{P^2}{(q+qP^3\beta)^{2/3}} + q^{\varepsilon}(q+qP^3\beta)\right) \d\beta \\
	&\ll P^{\frac13+\frac{4\phi}3} + P^{\frac13+\phi+\frac{\psi}3} + P^{2\phi+\psi+\varepsilon} + P^{\phi+2\psi+\varepsilon} \ll P^{1-\delta}
\end{align*}
whenever $\phi+\psi<\frac{1}{2}-\delta$. Hence by Parseval's identity, for all $k\ge 3$ we have
\begin{align}\label{eq:C9}
	\int_{\mathfrak m} |f_k(\alpha)|^2 \d\alpha = P+ O(P^{1-\delta}),
\end{align}
which gives the desired conclusion. Similar arguments may easily be implemented for multidimensional exponential sums, and it is also quite feasible to consider higher moments than the second.

\par
\bigskip 

{\textbf{Acknowledgements.}} This paper emerged as part of a group discussion at the Heilbronn Focused Research Workshop on Decoupling and Efficient Congruencing, University of Bristol, 17th--21st June 2019. The authors are very grateful to the Heilbronn Institute for Mathematical Research, as well as the European Union's Horizon 2020 research and innovation programme via grant agreement No.~695223, for funding this workshop, and to the University of Bristol for its hospitality. The work was further facilitated by a visit of three of the authors to Oberwolfach in November 2019. In addition, JB was supported by Starting Grant no.~2017-05110 of the Swedish Science Foundation (Vetenskapsr{\aa}det), and GS was supported by Simons Investigator Grant no.~376201 in the name of Ben Green.


\begin{thebibliography}{99}
	
	\bibitem[Bak86]{RB86} R. C. Baker, Diophantine inequalities, London Math. Soc. Monographs, New Series, vol. 1.  The Clarendon Press, Oxford, pp. xii+275, 1986.
	
	\bibitem[Ber96]{Ber} M.~Berry, \emph{Quantum fractals in boxes}, J. Phys. A: Math. Gen. \textbf{29} (1996), 6617--6629.
	
	\bibitem[BDG16]{BDG} J. Bourgain, C. Demeter and L. Guth, \emph{Proof of the main conjecture in Vinogradov's mean value theorem for degrees higher than three}, Ann. of Math. (2) \textbf{184} (2016) 633--682.
	
	\bibitem[BrDa09]{BD} J. Br\"udern and D. Daemen, \emph{Imperfect mimesis of Weyl sums}, Int. Math. Res. Not. (IMRN) 2009, No. 16, 3112--3126.
	
	\bibitem[BrRo15]{BR} J. Br\"udern and O. Robert, \emph{Rational points on linear slices of diagonal hypersurfaces}, Nagoya Math. J. \textbf{218} (2015), 51--100.

	\bibitem[ChSh]{ChS} C. Chen and I. E. Shparlinski, \emph{New bounds of Weyl sums}, Int. Math. Res. Not. (IMRN), to appear. 
	
	\bibitem[Dae10]{Dae} D. Daemen, \emph{The asymptotic formula for localized solutions in Waring's problem and approximations to Weyl sums}, Bull. London Math. Soc. \textbf{42} (2010) 75--82.
	
	\bibitem[ErSh19]{ES} M. B. Erdo\u gan and G. Shakan, \emph{Fractal solutions of dispersive partial differential equations on the torus},  Selecta Math. (N.S.) \textbf{25} (2019), no. 1, Art. 11.
	
	\bibitem[ErTz16]{ET} M. B. Erdo\u gan and N. Tzirakis, {\it  Dispersive Partial Differential Equations: Wellposedness and Applications}, London Mathematical Society Student Texts {86}, Cambridge University Press, 2016.
	
	\bibitem[HB10]{HB10} D. R. Heath-Brown, \emph{Bounds for the cubic Weyl sums}, J. Math. Sciences \textbf{171} (2010), 813--823.

	\bibitem[HuWo]{HW} K. Hughes and T. D. Wooley, \emph{Discrete restriction for $(x,x^3)$ and related topics}, arXiv:1911.12262.


	\bibitem[OsCh12]{Os} K. I. Oskolkov and M. A. Chakhkiev, \emph{Traces of the discrete Hilbert transform with quadratic phase} (Russian), Tr. Mat. Inst. Steklova 280 (2013), Ortogonalnye Ryady, Teoriya Priblizheni i Smezhnye Voprosy, 255--269; translation in Proc. Steklov Inst. Math. 280 (2013), no. 1, 248--262.

	\bibitem[Sch76]{WS76} W. M. Schmidt, Equations over finite fields -- An elementary approach. Lecture Notes in Mathematics \textbf{536}, Springer, Berlin, 1976.
	
	\bibitem[Sch80]{WS80} W. M. Schmidt, Diophantine Approximation. Lecture Notes in Mathematics, \textbf{785}, Springer, Berlin, 1980.
	
	\bibitem[Tal36]{Tal} H. F. Talbot, \emph{Facts related to optical science}, No. IV, Philo. Mag. \textbf{9} (1836), 401--407.
	
	\bibitem[Tit86]{Titch} E. C. Titchmarsh, The Theory of the Riemann Zeta-Function. 2nd
	Ed., revised by D. R. Heath-Brown, Oxford University Press, 1986.
	
	\bibitem[Vau97]{V:HL} R. C. Vaughan, The Hardy--Littlewood method. Cambridge Tracts in Mathematics \textbf{125}, Cambridge University Press, 1997.
	
	\bibitem[Vau09]{V:09} R. C. Vaughan, \emph{On generating functions in additive number theory, I}, in Analytic Number Theory, Essays in Honour of Klaus Roth, edited by W. W. L. Chen, W. T. Gowers, H. Halberstam, W. M. Schmidt, R. C. Vaughan, Cambridge University Press, 2009.
	
	\bibitem[VaWo98]{VW} R. C. Vaughan and T. D. Wooley, \emph{On the distribution of generating
	functions}, Bull. London Math. Soc. \textbf{30} (1998), 113--122.

	\bibitem[Woo15]{W:15} T. D. Wooley, \emph{Mean value estimates for odd cubic Weyl sums}, 	Bull. Lond. Math. Soc. \textbf{47} (2015), no. 6, 946--957.

	\bibitem[Woo16a]{W:16} T. D. Wooley, \emph{Perturbations of Weyl Sums}, Int. Math. Res. Not. (IMRN) \textbf{2016} (2016), No. 9, 2632--2646.

	\bibitem[Woo16b]{W:mvmt3} T. D. Wooley, \emph{The cubic case of the main conjecture in Vinogradov's mean value theorem}, Adv. Math. \textbf{294} (2016), 532--561.
	
	\bibitem[Woo19]{W:nec} T. D. Wooley, \emph{Nested efficient congruencing and relatives of Vinogradov's mean value theorem},  Proc. Lond. Math. Soc. (3) \textbf{118} (2019), no. 4, 942--1016.
	
	
	
\end{thebibliography}
\end{document}